\documentclass[11pt]{article}
\usepackage[english]{babel}
\usepackage{amsmath,amsthm}
\usepackage[T1]{fontenc}
\usepackage{amsfonts}
\usepackage{mathrsfs}
\usepackage{color}
\usepackage{bbm}
\usepackage{enumerate}
\usepackage{amsrefs}
\usepackage{mathtools}
\usepackage{hyperref}

\newtheorem{theorem}{Theorem}[section]

\newtheorem{lemma}[theorem]{Lemma}
\newtheorem{proposition}[theorem]{Proposition}

\theoremstyle{definition}

\theoremstyle{remark}

\numberwithin{equation}{section}
\begin{document}
	
	\def\Pro{{\mathbb{P}}}
	\def\E{{\mathbb{E}}}
	\def\e{{\varepsilon}}
	\def\veps{{\varepsilon}}
	\def\ds{{\displaystyle}}
	\def\nat{{\mathbb{N}}}
	\def\Dom{{\textnormal{Dom}}}
	\def\dist{{\textnormal{dist}}}
	\def\R{{\mathbb{R}}}
	\def\O{{\mathcal{O}}}
	\def\T{{\mathcal{T}}}
	\def\Tr{{\textnormal{Tr}}}
	\def\sgn{{\textnormal{sign}}}
	\def\I{{\mathcal{I}}}
	\def\A{{\mathcal{A}}}
	\def\H{{\mathcal{H}}}
	\def\S{{\mathcal{S}}}
	
	\title{Solutions to the stochastic heat equation with polynomially growing multiplicative noise do not explode in the critical regime}%
	\author{M. Salins\\ Boston University \\ msalins@bu.edu }
	%\thanks{}%
	%\date{}%
	% ----------------------------------------------------------------
	\maketitle
	
	\begin{abstract}
		We investigate the finite time explosion of the stochastic heat equation $\frac{\partial u}{\partial t} = \Delta u(t,x) + \sigma(u(t,x))\dot{W}(t,x)$ in the critical setting where $\sigma$ grows like $\sigma(u) \approx C(1 + |u|^\gamma)$ and $\gamma = \frac{3}{2}$. Mueller previously identified $\gamma=\frac{3}{2}$ as the critical growth rate for explosion and proved that solutions cannot explode in finite time if $\gamma< \frac{3}{2}$ and solutions will explode with positive probability if $\gamma>\frac{3}{2}$. This paper proves that explosion does not occur in the critical $\gamma=\frac{3}{2}$ setting.
	\end{abstract}

\section{Introduction} \label{S:intro}
We investigate whether solutions to the stochastic heat equation explode in finite time. The equation is
\begin{equation} \label{eq:SPDE}
	\begin{cases}
	\frac{\partial u}{\partial t}(t,x) = \Delta u(t,x) + \sigma(u(t,x)) \dot{W}(t,x), & x \in [-\pi,\pi], t>0\\
	u(t,-\pi) = u(t,\pi), & t>0\\
	u(0,x) = u_0(x) \text{ bounded and periodic}.
\end{cases}
\end{equation}
where $\sigma$ is locally Lipschitz continuous and satisfies the critical superlinear growth restriction that there exists $C>0$ such that for all $u \in \mathbb{R}$
\begin{equation} \label{eq:sigma-growth}
	|\sigma(u)| \leq C( 1 + |u|^{\frac{3}{2}})
\end{equation}
The spatial domain is $D =[-\pi,\pi]$ and we impose periodic boundary conditions. The stochastic noise $\dot{W}$ is spacetime white noise and the initial data $u_0(x)$ is a bounded, continuous, periodic function.

In \cite{mueller-1991,mueller-1998,ms-1993,mueller-2000}, Mueller and Sowers proved that the polynomial growth rate of $|u|^{\frac{3}{2}}$ is critical in the sense that if $\sigma(u) \leq C(1 + |u|^\gamma)$ for some $C>0$ and $\gamma< \frac{3}{2}$, the solution to the SPDE \eqref{eq:SPDE} cannot explode in finite time. If $\sigma(u) \geq c|u|^\gamma$ for some $c>0$ and $\gamma>\frac{3}{2}$ then solutions will explode with positive probability. The question of whether solutions can explode in finite time in the critical case of $\gamma=3/2$ was left unsolved. In this paper we prove that solutions cannot explode in the critical regime where $\gamma = \frac{3}{2}$.

Mueller's results have been generalized to other  settings including fractional heat equations \cite{bezdek-2018,fln-2019}, nonlinear Schr\"odinger equation \cite{bd-2002}, and stochastic wave equation \cite{mueller-1997}.  More recently, researchers have investigated the effects that adding superlinear deterministic forcing terms $f(u(t,x))$ to the right-hand side of \eqref{eq:SPDE} has on the finite time explosion of the stochastic heat equation \cite{bg-2009,dkz-2019,salins-2022,fn-2021,sz-2022,ch-2023,fp-2015,lz-2022,av-2023}. Similar explosion problems have been investigated for the stochastic wave equation \cite{fn-wave-2022,ms-wave-2021}. Interestingly, in \cite{dkz-2019}, for example, the authors prove that if the additional force $f(u)$ grows like $|u| \log(|u|)$ then $\sigma$ can grow like $|u| (\log(|u|))^{\frac{1}{4}}$ and solutions will never explode -- a much slower growth rate than the allowable $|u|^{\frac{3}{2}}$ growth rate when $f \equiv 0$. This $|u|(\log|u|)^{\frac{1}{4}}$ growth rate is not known to be optimal and it will be interesting to see if explosion can occur is $\sigma(u) \approx (1 + |u|^\frac{3}{2})$ when $f$ grows superlinearly. 

In the opposite setting where $f$ is strongly dissipative, $\sigma$ can grow faster than $|u|^\frac{3}{2}$ and solutions will not explode because the dissipative forcing counteracts the expansion due to the noise \cite{salins-2022-dissip}. Specifically, in this space-time white noise setting, if $f(u)\text{sign}(u) \leq -\mu |u|^\beta$ for some $\beta>3$, then $\sigma$ can grow like $C(1 +|u|^\gamma)$ for any $\gamma < \frac{\beta+3}{4}$ and solutions will not explode. In the setting of the current paper, $f \equiv 0$ and the maximal allowable growth rate for $\sigma$ is \eqref{eq:sigma-growth}. 
	
The mild solution to \eqref{eq:SPDE} is defined to be the solution to the integral equation
\begin{equation} \label{eq:mild-intro}
	u(t,x) = \int_D G(t,x-y)u(0,y)dy + \int_0^t \int_D G(t-s,x-y) \sigma(u(s,y))W(dyds)
\end{equation}
where $G(t,x)$ is the fundamental solution to the heat equation on $D$ with periodic boundary conditions. Because $\sigma$ is locally Lipschitz continuous, standard localization arguments prove that there exists a unique \textit{local} mild solution to \eqref{eq:mild-intro} that exists until the explosion time
\begin{equation}
	\tau^\infty_\infty : =\sup_{n>0} \tau^\infty_n
\end{equation}
where
\begin{equation}
	\tau^\infty_n :=  \inf\left\{t>0: \sup_{x \in D} |u(t,x)| \geq n\right\}.
\end{equation}
A local mild solution \textit{explodes in finite time} if $\tau^\infty_\infty < \infty$. A local mild solution is called a \textit{global} mild solution if the solution never explodes with probability one, $\Pro(\tau^\infty_\infty=\infty)=1$.

The main result of this paper, Theorem \ref{thm:main} proves that when \eqref{eq:sigma-growth} is satisfied, the mild solution is global.
\begin{theorem} \label{thm:main}
	Assume that the initial data is $x \mapsto u(0,x)$ is a bounded, continuous, periodic function on $[-\pi,\pi]$ and assume that $\sigma$ is locally Lipschitz continuous and satisfies \eqref{eq:sigma-growth}. Then there exists a unique global mild solution to \eqref{eq:SPDE}.
\end{theorem}
	
The method of proof is inspired by \cite{mueller-1991,mueller-1998}, but a new strategy is needed to prove non-explosion in the critical $\gamma =\frac{3}{2}$ setting. The first step is to prove that the $L^1$ norm of the solutions cannot explode. The fact that the $L^1$ norm cannot explode is easiest to see in the special case where $u(t,x)\geq0$ for all $t>0$ and all $x \in D$. Imposing the additional assumptions that $\sigma(0)=0$ and $u(0,x)\geq 0$, for example, would imply that $u(t,x)\geq0$ for all $t>0$ with probability one because of the comparison principle \cite{mueller-1991-support,kotelenez-1992}. In the case of a positive solution, formally integrating mild solutions in space indicates that 
\begin{equation}
	|u(t)|_{L^1} = \int_D u(t,x)dx = \int_D u(0,x)dx + \int_0^t \int_D \sigma(u(s,x))W(dsdx)
\end{equation}
is a nonnegative one-dimensional martinagle and therefore cannot explode in finite time. This argument can be made rigorous with stopping times. 
%The heat kernel $G(t,x)$, like the heat kernel on $\mathbb{R}$  behaves like $G(t,x)\leq C t^{-\frac{1}{2}}$ for all $t>0$ and $x \in D$. This means that if $u_0 \in L^1$, the solution to the linear determinsitic heat equation decays like
%\begin{equation} \label{eq:unif-decay}
%	\int_D G(t,x-y)u_0(y)dy \leq Ct^{-\frac{1}{2}}|u_0|_{L^1}.
%\end{equation}
%Because the random field solutions stay uniformly bounded in $L^1$ norm, their $L^\infty$ norms will decay uniformly quickly. Mueller constructed a sequence of stopping times that keep track of when the $L^\infty$ norm of the mild solution reaches the level $2^m$ for $m \in \mathbb{Z}$ \cite{mueller-1991}. Using \eqref{eq:unif-decay} and exponential tail estimates on the stochastic term in \eqref{eq:mild-intro}, he proved that when the $L^\infty$ norm is $2^m$ for large $m$, the probability of reaching $2^{m+1}$ before falling to $2^{m-1}$ is less than $1/2$. Therefore, the $L^\infty$ norms can be compared to a recurrent random walk and the $L^\infty$ norm cannot explode. This proof works when $\sigma(u)  \leq C(1  + |u|^\gamma)$ for $\gamma< \frac{3}{2}$, but the proof fails in the critical $\gamma=\frac{3}{2}$ case.

In the more general setting of this paper, where solutions $u(t,x)$ may take both  positive and negative values, we follow the ideas of \cite{mueller-1998} to construct nonnegative processes $v(t,x)$ and $v_-(t,x)$ that almost surely dominate $u(t,x)$ in the sense that 
\begin{equation}
	-v_-(t,x) \leq u(t,x) \leq v(t,x).
\end{equation}
Specifically, let $\alpha>3$ and let $f(u) = u^{-\alpha}$. Let $v(t,x)$ be the mild solution to
\begin{equation}
	\frac{\partial v}{\partial t} = \Delta v(t,x) + f(v(t,x)) + \sigma(v(t,x))\dot{W}(t,x)
\end{equation}
with initial data $v(0,x) = \max\{u(0,x),1\}$. Corollary 1.1 of \cite{mueller-1998} proves that solutions $v(t,x)$ remain nonnegative. The comparison principle of \cite[Theorem 2.5]{kotelenez-1992} proves that $u(t,x) \leq v(t,x)$ with probability one. $v_-(t,x)$ is constructed similarly. Then if we can prove that $v(t,x)$ and $v_-(t,x)$ do not explode to $+\infty$ in finite time, then $u(t,x)$ cannot explode in finite time either. 

We construct several stopping times to analyze these solutions. 
for any $n \in \mathbb{N}$ we define the $L^\infty$ stopping times
\begin{align}
	&\tau^\infty_n = \inf\{t>0: \sup_{x \in D} v(t,x) \geq n\},\\
	&\tau^\infty_\infty = \sup_n \tau^\infty_n.
\end{align}
The solution explodes in finite time if and only if $\tau^\infty_\infty <\infty$. Therefore, the goal of this paper is to prove that $\Pro(\tau^\infty_\infty = \infty) =1$.
Becuase $f$ is unbounded near $0$, we also need to define  the infimum stopping times for any $\e>0$,
\begin{equation}\label{eq:tau-inf-v}
	\tau^{\inf}_\e = \inf\left\{t\in [0,\tau^\infty_\infty): \inf_{x \in D} v(t,x) \leq \e \right\}
\end{equation}
Because $f(u)$ is Lipschitz continuous on $[\e,\infty)$ for any $\e>0$ and $\sigma(u)$ is Lipschitz continuous for $u \in [0,n]$ for any $n>0$, there exists a local mild solution for $v(t,x)$ until the time $\tau^\infty_\infty \wedge \tau^{\inf}_0$ where $\wedge$ denotes the minimum.

Corollary 1.1 of \cite{mueller-1998} proves that $v(t,x)$ never hits zero. Specifically, for any $T>0$,
\begin{equation}
	\lim_{\e \to 0} \Pro(\tau^{\inf}_\e \leq T \wedge \tau^\infty_\infty) = 0.
\end{equation}

For $M>0$, we define the $L^1$ stopping times
\begin{equation} \label{eq:tau1-v}
	\tau^1_M := \inf\{t \in [0, \tau^\infty_\infty): |v(t)|_{L^1} >M\}
\end{equation} 
%
%
%
%The $L^1$ norm of $v(t,x)$ is %not a martingale because of the positive drift $f(u)$, but $\int_D v(t,x)dx$ is 
%a nonnegative submartingale. Specifically, for $n>0$ big and $\e>0$ small define
%\begin{align}
%	&I_{n,\e}(t):= \int_D v(t \wedge \tau^{\inf}_\e \wedge \tau^\infty_n,x)dx \nonumber\\
%	&= I_{n,\e}(0) + \int_0^{t \wedge \tau^{\inf}_\e \wedge \tau^\infty_n} \int_D f(v(s,y))dyds + \int_0^{t \wedge \tau^{\inf}_\e \wedge \tau^\infty_n}\sigma(v(s,y))W(dyds).
%\end{align}
and we prove that  the $L^1$ norm $\int_D v(t \wedge \tau^{\inf}_\e \wedge \tau^\infty_n,x)dx$ is a submartingale.
Using Doob's submartingale inequality we can prove that for any $T>0$ and $\e>0$ the $L^1$ norm cannot explode before $T \wedge \tau^{\inf}_\e$. The estimates are independent of $n$.

The novel observation, which is necessary to extend Mueller's results to the critical case where $\gamma = \frac{3}{2}$, is that we can show that the expected value of the \textit{quadratic variation}  of the $L^1$ norm is also bounded in a way that is independent of $n$. We prove in Lemma \ref{lem:L1} that 
\begin{align} \label{eq:quad-var-intro}
	&\E \int_0^{\tau^{\inf}_\e \wedge \tau^1_M } (\sigma(v(s,y)))^2dyds \leq M^2,
\end{align} 
an estimate that is independent of $n$ and $\e$.

In Section \ref{S:moment}, we prove an improved $L^\infty$ moment bound on the stochastic convolution, inspired by \cite{ch-2023}, which may be of independent interest.
\begin{theorem} \label{thm:stoch-conv}
	Let $p>6$. Assume that $\varphi(t,x)$ is an adapted random field such that
	\begin{equation}
		\E \int_0^T \int_D|\varphi(t,x)|^pdxdt < +\infty.
	\end{equation}
	Define the stochastic convolution
	\begin{equation}
		Z^\varphi(t,x) = \int_0^t \int_D G(t-s,x,y) \varphi(s,y)W(dyds).
	\end{equation}
	For any $p>6$ there exists $C_p>0$, independent of $T>0$ and $L>0$, such that
	\begin{equation} \label{eq:Linfty-moment}
		\E \sup_{t \in [0,T]} \sup_{x \in D} |Z^\varphi(t,x)|^p
		\leq C_p T^{\frac{p}{4} - \frac{3}{2}} \E\int_{0}^{T} \int_{D} |\varphi(s,y)|^p dy ds.
	\end{equation}
\end{theorem}
We remark  that in the case where there exists $L>0$ such that
\begin{equation}
	\Pro\left(\sup_{t \in [0,T]} \sup_{x \in D} |\varphi(t,x)| \leq L\right) = 1,
\end{equation}
an obvious upper bound of \eqref{eq:Linfty-moment} is 
\begin{equation}
	C_p T^{\frac{p}{4} - \frac{3}{2}}  \E\int_{0}^{T} \int_{D} |\varphi(s,y)|^p dy ds \leq C_pL^{p} T^{\frac{p}{4} - \frac{1}{2}}.
\end{equation} 
This looser upper bound can be used to prove non-explosion in the subcritical $\gamma< \frac{3}{2}$ regime. Unfortunately,
this looser bound will not be helpful when we prove the main non-explosion result in the critical setting and we will need the tighter upper bound \eqref{eq:Linfty-moment}.
%Importantly, we will apply this result to $\varphi(t,x) = \sigma(v(t , x ))\mathbbm{1}_{\{t \leq \tau^1_M \wedge \tau^{\inf}_\e\}}$. The integral $$\E\int_{0}^{\tau^1_M \wedge \tau^{\inf}_\e} \int_{D} |\sigma(v(s,y))|^2 dy ds $$ in \eqref{eq:Linfty-moment}  is bounded because it is the quadratic variation from \eqref{eq:quad-var-intro}.

We then define a sequence of stopping times $\rho_n$ that keep track of when the $|v(t)|_{L^\infty}$ doubles or halves. The stopping times are defined so that $|v(\rho_n)|_{L^\infty} = 2^m$ for some $m \in \mathbb{N}$. Using all of the estimates mentioned above we can prove that for any $\e>0$ and $ M>0$, the $L^\infty$ norm $|v(\rho_n)|_{L^\infty}$ can only double a finite number of times before the time $ \tau^{\inf}_\e \wedge \tau^1_M$. This estimate relies on estimates of the quadratic variation of the $L^1$ norm \eqref{eq:quad-var-intro}, which were not required in the subcritical setting. Therefore, for any $\e>0$ and $M>0$, the explosion time
\begin{equation}
	\tau^\infty_\infty> ( \tau^{\inf}_\e \wedge \tau^1_M )\text{ with probability one.}
\end{equation}
Taking the limit as $M \to \infty$ and  $\e \to 0$, we can prove that explosion cannot occur in finite time.

In Section \ref{S:notation} we introduce some notations and recall the properties of the heat kernel. In Section \ref{S:comparison} we introduce the positive solutions $v(t,x)$ and $v_-(t,x)$ and prove that they dominate $u(t,x)$. In Section \ref{S:L1} we prove that the $L^1$ norm of the solutions its quadratic variation remain finite in a way that does not depend on the $L^\infty$ norm of the solutions. In Section \ref{S:moment} we prove the stochastic convolution moment bound Theorem \ref{thm:stoch-conv}. Finally in Section \ref{S:non-explosion} we prove that $v(t,x)$ cannot explode in finite time.

\section{Some notations and definitions} \label{S:notation}
The spatial domain $D=[-\pi,\pi]$. 

Let $L^p := L^p(D)$, $p\geq 1$ denote the standard $L^p$ spaces on $D$ endowed with the norms
\begin{align}
	&|\varphi|_{L^p} := \left(\int_D|\varphi(y)|^p\right)^{\frac{1}{p}},  \ \ \  p \in [1,\infty), \\
	&|\varphi|_{L^\infty} :=  \sup_{x \in D} |\varphi(x)|.
\end{align}

The driving noise $\dot{W}$ is a space-time white noise defined on a filtered probability space $(\Omega, \mathcal{F}, \mathcal{F}_t, \Pro)$. This means that for any non-random $\psi, \varphi \in L^2([0,T]\times D)$, 
\[\int_0^T \int_D \varphi(s,y)W(dyds) \text{ and } \int_0^T \int_D \psi(s,y)W(dyds) \]
are mean-zero Gaussian random variables with covariance

\begin{align}
	&\E \left( \int_0^T \int_D \varphi(s,y)W(dyds) \right)\left(\int_0^T \int_D \psi(s,y)W(dyds) \right) \nonumber\\
	&= \int_0^T \int_D \varphi(s,y)\psi(s,y)dyds.
\end{align}

If $\varphi(t,x)$ is an $\mathcal{F}_t$-adapted process then the stochastic integral $$\int_0^t \int_D \varphi(s,y)W(dyds)$$ is an Ito-Walsh integral \cite{walsh-1986}.

The heat kernel on $D$ with periodic boundary is defined to be
\begin{equation} \label{eq:kernel}
	G(t,x) = \frac{1}{\sqrt{2\pi}} + \sum_{k=1}^\infty \sqrt{\frac{2}{\pi}} e^{-|k|^2t } \cos(kx).
\end{equation}

For any $\varphi \in L^2(D)$, $h(t,x) = \int_D G(t,x-y)\varphi(y)dy$ solves the linear heat equation $ \frac{\partial h}{\partial t} = \Delta h$ with initial data $h(0,x) = \varphi(x)$.

\begin{lemma} \label{lem:kernel-estimates}
  The heat kernel has the following properties.
  \begin{enumerate}
  	\item The heat kernel is nonnegative: $G(t,x) \geq0$ for all $t>0$, $x \in D$.
  	\item $|G(t,\cdot)|_{L^1} = \sqrt{2\pi}$.
  	\item There exists $C>0$ such that for any $t>0$,
  	\begin{equation}
  		|G(t,\cdot)|_{L^\infty} \leq C t^{-\frac{1}{2}},
  	\end{equation}
	\end{enumerate}
\end{lemma}

\begin{proof}
	The positivity of the heat kernel is a consequence of the comparison principle for linear heat equations. Specifically, let $\varphi: D \to \mathbb{R}$ be any nonnegative periodic function. $h(t,x) = \int_D G(t,x-y)\varphi(y)dy$ solves the heat equation and therefore satisfies the comparison principle. Therefore $h(t,x) \geq 0$ for all $t>0$ and $x \in D$ because $\varphi(x)\geq 0$ for all $x \in D$. This is true for any nonnegative $\varphi$, implying that $G(t,x)\geq 0$.
	
	The $L^1$ norm claim can be calculated exactly because $G(t,x)$ is nonnegative and  $\int_{-\pi}^\pi \frac{1}{\sqrt{2\pi}} = \sqrt{2\pi}$ and $\int_{-\pi}^\pi \cos(kx)dx =0$ for $k \geq 1$.
	
	For the $L^\infty$ norm we notice that for any $t>0$ and  $x \in D$ 
	\begin{align}
		& |G(t,x)| \leq G(t,0) \leq\sqrt{\frac{2}{\pi}} + \frac{1}{\sqrt{\pi}}\sum_{k=1}^\infty e^{-|k|^2 t}  \nonumber\\
		&\leq  \sqrt{\frac{2}{\pi}}+ \sqrt{\frac{1}{\pi}}\int_0^\infty e^{-|x|^2t}dx \nonumber\\
		&\leq  \sqrt{\frac{2}{\pi}}+ \frac{1}{2}t^{-\frac{1}{2}}
		\nonumber\\
		&\leq C t^{-\frac{1}{2}}.
	\end{align}
\end{proof}

Throughout the paper we use the notation $a \wedge b = \min\{a,b\}$ and $C$ denotes an arbitrary constant whose value may change from line to line.

\section{Comparison to positive solutions} \label{S:comparison}
%Formally, the spatial integral
%\begin{equation}
%	\int u(t,x)dx  = \int u(0,x)dx + \int_0^t \int \sigma(u(s,x))W(dxds)
%\end{equation}
%is a martingale. This claim can be made rigorous using stopping time arguments. If $u(t,x)$ were nonnegative, then $\int u(t,x)dx$ would be the $L^1$ norm of solutions and using martingale arguments we could prove that the $L^1$ norm cannot explode in finite time.
%
%For example, if we additionally assumed that $\sigma(0) = 0$ and that the initial data $u(0,x)\geq 0$, then we can guarantee that solutions stay positive. More generally, 
We follow the arguments of \cite{mueller-1998} to construct nonnegative stochastic processes that dominate $u(t,x)$. Specifically, let $f(u) = u^{-\alpha}$ for some $\alpha>3$.

Let $v(t,x)$ be the solution to
\begin{equation} \label{eq:v}
	\frac{\partial v}{\partial t} (t,x) = \Delta v(t,x) + f(v(t,x)) + \sigma(v(t,x))\dot{W}(t,x)
\end{equation}
with initial data
\begin{equation}
	v(0,x) = \max\{u(0,x),1\}
\end{equation}

and let $v_-(t,x)$ be the solution to
\begin{equation} \label{eq:v-}
	\frac{\partial v_-}{\partial t} (t,x) = \Delta v_-(t,x) + f(v_-(t,x)) + \sigma(-v_-(t,x))\dot{W}(t,x)
\end{equation}
with initial data
\begin{equation}
	v_-(0,x) = \max\{-u(0,x),1\}.
\end{equation}

$v(t,x)$ and $v_-(t,x)$ have the same properties. For this reason, we only prove results for $v(t,x)$ because the proofs for $v_-(t,x)$ are identical.

We now recall the standard arguments for the construction of the unique \textit{local mild solution} to \eqref{eq:v}. For any $\e>0$ define  
\begin{equation}
	f_\e(u) = (\max\{\e,u\})^{-\alpha}.
\end{equation}
Notice that for any $\e>0$, $f_\e$ is globally Lipschitz continuous. For any $n>0$ define
\begin{equation}
	\sigma_n(u) = \begin{cases}
		\sigma(-n) & \text{ if } u<-n \\
		\sigma(u) & \text{ if } u \in [-n,n] \\
		\sigma(n) & \text{ if } u> n
	\end{cases}.
\end{equation}
For each $n>0$, $\sigma_n$ is globally Lipschitz continuous. Therefore by standard arguments \cite{dalang-1999,cerrai-2003,walsh-1986} for any $\e>0$ and $n>0$ there exists a unique global mild solution $v_{\e,n}$ solving
\begin{align}
	v_{\e,n}(t,x) = &\int_D G(t,x-y)v(0,y)dy + \int_0^t \int_D G(t-s,x-y)f_\e(v_{\e,n}(s,y))dyds \nonumber\\
	&+ \int_0^t \int_D G(t-s,x-y) \sigma_n(v_{\e,n}(s,y))W(dyds)
\end{align}
where $G(t,x)$ is the heat kernel defined in \eqref{eq:kernel}. 

For any $0<\e<n$ define the stopping times
\begin{equation}
	\tilde{\tau}_{\e,n} : = \inf\left\{t>0: \inf_{x \in D}v_{\e,n}(t,x) < \e \text{ or } \sup_{x \in D}v_{\e,n} >n  \right\}
\end{equation}
For any $0< \e_2 < \e_1<n_1< n_2$, the functions $f_{\e_1}(u) = f_{\e_2}(u)$ and $\sigma_{n_1}(u) = \sigma_{n_2}(u)$ for all $ u \in [\e_1,n_1]$.
Therefore, the uniqueness of solutions implies that these solutions are \textit{consistent} in the sense that if $0< \e_2 < \e_1<n_1< n_2$ then
\begin{equation}
	v_{\e_1,n_1}(t,x) = v_{\e_2,n_2}(t,x) \text{ for all } x \in D \text{ and } t \in [0,\tilde{\tau}_{\e_1,n_1}].
\end{equation}
We can, therefore, uniquely define the unique local mild solution by 
\begin{equation}
	v(t,x) := v_{\e,n}(t,x) \text{ for all } x \in D \text{ and } t \in [0,\tilde{\tau}_{\e,n}]
\end{equation}
and the local mild solution is well defined for all $t \in [0,\sup_{0<\e<n} \tilde{\tau}_{\e,n}]$
and  solves the integral equation
\begin{align} \label{eq:v-mild}
	v(t,x) = &\int_D G(t,x-y)v(0,y)dy + \int_0^t \int_D G(t-s,x-y)f(v(s,y))dyds \nonumber\\
	&+ \int_0^t \int_D G(t-s,x-y)\sigma(v(s,y))W(dyds).
\end{align}
The construction of $v_-(t,x)$ is identical so we do not repeat the proof.

Define the infimum stopping times for $\e \in (0,1)$
\begin{align} \label{eq:tau-inf}
	&\tau^{\inf}_\e : = \inf\left\{t>0: \inf_{x \in D} v(t,x) < \e\right\}\\
	&\tau^{\inf}_{\e,-} : = \inf\left\{t>0: \inf_{x \in D} v_-(t,x)<\e\right\}
\end{align}
and the $L^\infty$ stopping times for $n>1$
\begin{align}
	&\tau^\infty_n: = \inf\left\{t>0: \sup_{x \in D} v(t,x) > n\right\} \label{eq:tau-inf-L}\\
	&\tau^\infty_\infty :  = \sup_{n>0} \tau^\infty_n \label{eq:tau-inf-inf}.
\end{align}

\begin{align}
	&\tau^\infty_{n,-}: = \inf\left\{t>0: \sup_{x \in D}v_-(t,x) > n\right\}\\
	&\tau^\infty_{\infty,-} :  = \sup_{n>0} \tau^\infty_{n,-}.
\end{align}

The comparison principle of \cite[Theorem 2.5]{kotelenez-1992} guarantees that the following holds.
\begin{proposition} \label{prop:comparison}
	With probability one
	\begin{equation}
		-v_-(t,x) \leq u(t,x) \leq v(t,x)
	\end{equation}
  for all $t \in [0, \tau^{\inf}_0 \wedge \tau^{\inf}_{0,-}\wedge  \tau^\infty_\infty \wedge \tau^\infty_{\infty,-}]$ and for all $x \in [-\pi,\pi]$.
\end{proposition}
\begin{proof}
	The comparison principle of \cite{kotelenez-1992} is stated for heat equations with globally Lipschitz continuous $f(v)$ and $\sigma(v)$. But $f(v)$ and $\sigma(v)$ are both Lipschitz continuous for $v \in [\e,n]$ for any $0< \e< n< \infty$. Therefore, with probability one,
	\begin{equation}
		-v_-(t,x) \leq u(t,x) \leq v(t,x)
	\end{equation}
	for all $t \in [0, \tau^{\inf}_\e \wedge \tau^{\inf}_{\e,-} \wedge\tau^\infty_n \wedge\tau^\infty_{n,-}]$. Taking the limit as $\e \to 0$ and $n \to \infty$ proves the result.
\end{proof}

Corollary 1.1 of \cite{mueller-1998} proves that the $f(u) = u^{-\alpha}$ forcing and the nonnegative initial data of $v(0,x)$ prevent $v(t,x)$ from becoming negative. We restate this result below. 

%Notice that $f(u) = u^{-\alpha}$ is Lipschitz continuous for $u \in [\e,\infty)$ for any $\e>0$ and $\sigma$ is Lipschitz continuous for $u \in [0,L]$ for any $L>0$. Therefore there exists a unique local mild solution $v(t,x)$ to \eqref{eq:v} until the time $\tau^{\inf}_0 \wedge \tau^\infty_\infty$ and an analogous claim holds for $v_-(t,x)$. Corollary 1.1 of Mueller 1998 tells us $v(t,x)$ and $v_-(t,x)$ cannot ever hit zero before the explosion time. In our language this means that that $\tau^{\inf}_\e$ must be larger than $T \wedge \tau_L$ for any $T>0$ and $L>0$ for small $\e>0$.  We state this result below using our terminology. 

\begin{proposition}[Corollary 1.1 of \cite{mueller-1998}] \label{prop:v-positive}
	For any $T>0$ 
	\begin{equation}
		\lim_{\e \to 0}\Pro\left(\inf_{t \in [0,T\wedge \tau^\infty_\infty]}\inf_{x\in D} v(t,x)<\e \right)=0.
	\end{equation}
\end{proposition}

We will prove that under the assumptions of Theorem \ref{thm:main}, the solutions of $v(t,x)$ cannot explode in finite time. Because $v_-(t,x)$ satisfies the same assumptions as $v(t,x)$, $v_-(t,x)$ cannot explode in finite time either.  

\begin{theorem} \label{thm:v-no-explode}
	Let $v(t,x)$, $v_-(t,x)$ be the local mild solutions to \eqref{eq:v} and \eqref{eq:v-}. Then both $\tau^\infty_\infty = \infty$ and $\tau^\infty_{\infty,-}=\infty$ with probability one.
\end{theorem}

We will prove Theorem \ref{thm:v-no-explode} in Section \ref{S:non-explosion}. Then the comparison principle, Proposition \ref{prop:comparison}, guarantees that $u(t,x)$ cannot explode in finite time. The main result of our paper, Theorem \ref{thm:main}, will hold once we prove Theorem \ref{thm:v-no-explode}. 

\begin{proof}[Proof of Theorem \ref{thm:main}, assuming that Theorem \ref{thm:v-no-explode} holds]
	By the comparison principle, Proposition \ref{prop:comparison}, 
	\begin{equation}
		-v_-(t,x) \leq u(t,x) \leq v(t,x)
	\end{equation}
	For  $t \in [0, \tau^{\inf}_0 \wedge \tau^{\inf}_{0,-}\wedge  \tau^\infty_\infty \wedge \tau^\infty_{\infty,-}]$ and for all $x \in [-\pi,\pi]$.
	Theorem \ref{thm:v-no-explode} proves that $\tau^\infty_\infty = \tau^\infty_{\infty,-} = \infty$ with probability one. Proposition \ref{prop:v-positive} proves that 
	\begin{equation}
		\tau^{\inf}_0 \geq T \wedge \tau^\infty_\infty
	\end{equation}
	for any $T>0$. This is true for arbitrary $T>0$ and therefore
	$\tau^{\inf}_0 = \tau^{\inf}_{0,-} = \infty$.
	
	Therefore $u(t,x)$ can never explode.
\end{proof}

The rest of the paper is devoted to proving Theorem \ref{thm:v-no-explode}.

\section{The $L^1$ norm of  $v(t,x)$ } \label{S:L1}
The first step to prove that the solutions to $v(t,x)$ do not explode is to prove that the $L^1$ norms of solutions do not explode.

Let $v(t,x)$ be the nonnegative local mild solution to \eqref{eq:v}.
Define for $t \in [0, \tau^\infty_\infty]$
\begin{equation}
	 |v(t)|_{L^1}: = \int_D v(t,x)dx.
\end{equation}

Define the $L^1$ stopping times for $M>0$
\begin{equation} \label{eq:tau-1-M}
	\tau^1_M: = \inf\{t \in [0,\tau^\infty_\infty]: |v(t)|_{L^1} >M\}.
\end{equation}

\begin{lemma} \label{lem:L1}
	For any $T>0$, $\e>0$ and $M>0$,
	\begin{equation}
		\Pro \left(\sup_{t \in [0,T \wedge \tau^\infty_\infty \wedge \tau^{\inf}_\e] } |v(t)|_{L^1} > M\right) \leq \frac{|u(0)|_{L^1} + 2 \pi T \e^{-\alpha}}{M},
	\end{equation}
    In particular,
	\begin{equation}
		\Pro \left(\sup_{t \in [0,T \wedge \tau^\infty_\infty \wedge \tau^{\inf}_\e ] } |v(t)|_{L^1}< \infty\right)=1.
	\end{equation}
	Furthermore, for any $M>0$ and $\e>0$, the quadratic variation of $|v(t)|_{L^1}$ satisfies
	\begin{align} \label{eq:L3-bound}
		&\E \int_0^{   \tau^1_M \wedge \tau^{\inf}_\e } \int_D |\sigma(v(t,x))|^2 dyds  \leq M^2.
	\end{align}
\end{lemma}

\begin{proof}
	Let $n>0$ be big and $\e>0$ be small enough so that $\e < v(0,x) < n$ for all $x \in D$. Let
	\begin{align} 
		&I_{n,\e}(t ):=\int_Dv(t \wedge \tau^\infty_n \wedge \tau^{\inf}_\e,x)dx. 
	\end{align}
	The $\tau^{\inf}_\e$ stopping time guarantees that $v(t \wedge \tau^\infty_n \wedge \tau^{\inf}_\e,x)\geq \e$ so that $I_{n,\e}$ is the $L^1$ norm $ |v(t \wedge \tau^\infty_n \wedge \tau^{\inf}_\e)|_{L^1} $.
	Integrating the mild solution \eqref{eq:v-mild} and using the fact that $\int_DG(t,x-y)dx = 1$,
	\begin{align} \label{eq:L1-martingale}
        I_{n,\e}(t)= &\int_D v(0,y)dy + \int_0^{t\wedge \tau^\infty_n \wedge \tau^{\inf}_\e} \int_D f(v(s,x))dxds \nonumber\\
        &+ \int_0^{t\wedge \tau^\infty_n \wedge \tau^{\inf}_\e} \int_D   \sigma(v(s,y))W(dyds).
	\end{align}
	$I_{n,\e}(t)$ is  a nonnegative submartingale becuase $f(v)>0$ and because the stochastic integral in \eqref{eq:L1-martingale} is a martingale.
	Therefore, for any $M>0$ and $T>0$, by Doob's inequality 
	\begin{equation}
		%\Pro\left( \sup_{t>0} \int_Du_n(t,x)dx > M\right) \leq
		\Pro\left( \sup_{t \in [0,T]} I_{n,\e}(t) > M\right) \leq \frac{\E I_{n,\e}(T)}{M} \leq \frac{|v(0)|_{L^1} + 2\pi T\e^{-\alpha}}{M}
	\end{equation}
	because $f(v) \leq \e^{-\alpha}$ when $v>\e$, the length of $D=[-\pi,\pi]$ is $2\pi$, and because the expectation of the stochastic integral in \eqref{eq:L1-martingale} is zero.
	This bound does not depend on $n$.  Therefore,
	\begin{equation}
		\Pro \left(\sup_{t \in [0, T \wedge \tau^{\inf}_\e \wedge \tau^\infty_\infty]} \int_D v(t,x)dx > M \right) \leq \frac{|v(0)|_{L^1} + 2\pi T\e^{-\alpha}}{M}.
	\end{equation}
	Now take $M \uparrow \infty$ to see that
	\begin{equation}
		\Pro \left(\sup_{t \in [0, T \wedge \tau^{\inf}_\e \wedge \tau^\infty_\infty]} \int_D u(t,x)dx < \infty \right)=1.
	\end{equation}

	Now we apply Ito formula to \eqref{eq:L1-martingale}. For any $M>0$, $n>0$, $\e>0$,
	\begin{align}
		&\E(I_{n,\e}(t \wedge \tau^1_M))^2 \nonumber\\
		&= \E(I_{n,\e}(0))^2 + 2\E \int_0^{t \wedge \tau^\infty_n\wedge \tau^{\inf}_\e \wedge \tau^1_M} \int_D f(v(s,y))I_{n,\e}(s)dyds \nonumber\\
		&\qquad+ \E \int_0^{t \wedge \tau^\infty_n \wedge \tau^{\inf}_\e \wedge \tau^1_M} \int_D |\sigma(v(s,y))|^2dyds.
	\end{align}
	Each term on the right-hand side is nonnegative and $\E(I_{n,\e}(t \wedge \tau^1_M))^2 \leq M^2$ by the definition of $\tau^1_M$. Therefore,
	\begin{equation}
		\E \int_0^{t \wedge \tau^\infty_n \wedge \tau^{\inf}_\e \wedge \tau^1_M} \int_D |\sigma(v(s,y))|^2dyds \leq M^2.
	\end{equation} 
	This bound does not depend on $n$, $\e$, or $t$.
\end{proof}

\section{Moment estimates on the stochastic convolution} \label{S:moment}
In this section we prove the moment estimate Theorem \ref{thm:stoch-conv}.
\begin{proof}[Proof of Theorem \ref{thm:stoch-conv}]
	Let $p>6$ and assume that $\varphi(t,x)$ is adapted and
	\begin{equation}
		\E \int_0^T \int_D |\varphi(t,x)|^pdxdt < +\infty.
	\end{equation}
	We use Da Prato and Zabczyk's factorization method \cite[Theorem 5.10]{dpz}. Given $p>6$ let $\beta \in \left(\frac{3}{2p}, \frac{1}{4}\right)$ and define
	\begin{equation}
		Z^\varphi_\beta(t,x)= \int_0^t \int_D (t-s)^{-\beta}G(t-s,x,y) \varphi(s,y)W(dyds).
	\end{equation}
	Then
	\begin{equation}
		Z^\varphi(t,x) = \frac{\sin(\pi\beta)}{\pi} \int_0^t\int_D (t-s)^{\beta - 1} G(t-s,x,y)Z^\varphi_\beta(s,y)dyds.
	\end{equation}
	We can get supremum bounds on $Z^\varphi(t,x)$ by H\"older's inequality. This method was used, for example, by Chen and Huang \cite[Proof of Theorem 1.6]{ch-2023}.
	\begin{align}
		\sup_{t \in [0,T]} \sup_{x \in D} |Z^\varphi(t,x)|
		\leq &C\left(\int_0^t \int_D(t-s)^{\frac{(\beta-1)p}{p-1}}G^{\frac{p}{p-1}}(t-s,x-y)dyds\right)^{\frac{p-1}{p}} \nonumber\\
		&\times\left( \int_0^T\int_D |Z^\varphi_\beta(t,x)|^pdx\right)^{\frac{1}{p}}.
	\end{align}
	The integral 
	$$\int_D G^{\frac{p}{p-1}}(t-s,x-y)dy \leq |G(t-s)|_{L^1}|G(t-s)|_{L^\infty}^{\frac{p}{p-1} -1} \leq C (t-s)^{-\frac{1}{2(p-1)}}$$
	 because of Lemma \ref{lem:kernel-estimates}.
    Because we chose $p\beta > \frac{3}{2}$, it follows that $\frac{(\beta -1)p - \frac{1}{2}}{p-1} > -1$ and therefore
	\begin{align} \label{eq:Z-factored}
		&\E\sup_{t \in [0,T]} \sup_{x \in D} |Z(t,x)|^p\nonumber\\
		&\leq C \left(\int_0^t (t-s)^{\frac{(\beta -1)p - \frac{1}{2}}{p-1}}ds\right)^{p-1}\E\int_0^T \int_D |Z^\varphi_\beta(t,x)|^pdxdt \nonumber\\
		&\leq C T^{\beta p  - \frac{3}{2}} \E\int_0^T \int_D |Z^\varphi_\beta(t,x)|^pdxdt
	\end{align}
	It remains to estimate $\E\int_0^T \int_D |Z^\varphi_\beta(t,x)|^p dx dy$.
	By the BDG inequality,
	\begin{equation}
	\E |Z^\varphi_\beta(t,x)|^p \leq C_p \E\left( \int_0^t\int_D G^2(t-s,x-y)(t-s)^{-2\beta} |\varphi(s,y)|^2 dyds \right)^{\frac{p}{2}}.
	\end{equation}
	By Young's inequality for convolutions,
	\begin{align}
		&\int_0^T \int_D \E |Z^\varphi_\beta(t,x)|^pdxdt \nonumber\\
		&\leq C_p \left(\int_0^T \int_D G^{2}(s,y)s^{-2 \beta}dyds \right)^{\frac{p}{2}} \left(\int_0^T \int_D \E(|\varphi(s,y)|^p) dyds \right)
		\nonumber\\ 
		&\leq C_p \left(\int_0^T s^{-\frac{1}{2}-2 \beta}ds \right)^{\frac{p}{2}} \left(\int_0^T \int_D \E(|\varphi(s,y)|^p) dyds \right) \nonumber\\
		&\leq C_p T^{\frac{p}{4} - p \beta } \E\int_0^T \int_D |\varphi(s,y)|^pdyds.
	\end{align}
	In the second-to-last line we used Lemma \ref{lem:kernel-estimates} to estimate that $$\int_DG^2(s,y)dy \leq |G(s,\cdot)|_{L^\infty}|G(s,\cdot)|_{L^1} \leq  C s^{-\frac{1}{2}}.$$ % and
	%in the last line we used the assumptions that $|\varphi(t,x)| \leq L$ and $|\sigma(u)| \leq C( 1+ |u|^{\frac{3}{2}})$ to bound
	%\begin{equation}
	%	|\sigma(\varphi(s,y))|^p \leq C(1 + L)^{\frac{3p}{2} - 3}|\sigma(\varphi(s,y))|^2,
	%\end{equation} 
	Combining this with \eqref{eq:Z-factored} we conclude that
	\begin{equation}
		\E \sup_{t \in [0,T]} \sup_{x \in D} |Z(t,x)|^p \leq C T^{\frac{p}{4} - \frac{3}{2}}  \E\int_0^T \int_D |\varphi(s,y)|^p dyds.
	\end{equation}
\end{proof}

\section{Non-explosion of $v(t,x)$} \label{S:non-explosion}
Let $M>0$ and $\e>0$ be arbitrary. We will show that $v(t,x)$ cannot explode before time $ \tau^1_M \wedge \tau^{\inf}_\e$. After we prove this, we can take the limits as $M \to \infty$  in Lemma \ref{lem:L1} and $\e \to 0$ in Proposition \ref{prop:v-positive} to prove that explosion cannot ever occur.

%We follow the method of Mueller.

Fix $\e>0$, $M>0$ and define a sequence of stopping times $\rho_n$. These stopping times depend on the choices of $\e$, and $M$.
\begin{align}
	&\rho_{0}  =\inf\{t \in [0, \tau^{\inf}_\e \wedge \tau^1_M]: |v(t)|_{L^\infty} = 2^m \text{ for some } m \in \{1,2,3, ...\}\}.
\end{align}
Then if $|v(\rho_n)|_{L^\infty} = 2^m$ for $m \geq 2$ we define
\begin{align}
	&\rho_{n+1} =
		\inf\left\{t \in [\rho_{n}, \tau^{\inf}_\e \wedge \tau^1_M]: \begin{matrix} |v(t)|_{L^\infty}\geq 2^{m+1} \\
			\text{ or } |v(t)|_{L^\infty} \leq 2^{m-1} \end{matrix} \right\},\label{eq:rho-def}
\end{align}
and if $|v(\rho_n)|_{L^\infty} = 2$ then
\begin{align}
		&\rho_{n+1} = \inf\left\{t \in [\rho_{n}, \tau^{\inf}_\e \wedge \tau^1_M]: |v(t)|_{L^\infty}\geq 2^{2} \right\}.
\end{align}
These times keep track of how long it takes for the $L^\infty$ norm of the process to either double or half. We use the convention that $\rho_{n+1} =  \tau^{\inf}_\e \wedge \tau^1_M$ if the process stops doubling or halving after $\rho_n$.
 If the process were to explode before time $ \tau^{\inf}_\e \wedge \tau^1_M$, then the $L^\infty$ norm would need to double an infinite amount of times before that $\tau^{\inf}_\e \wedge \tau^1_M$. We prove that the process does not explode by proving that there can only be a finite number of times that the $L^\infty$ norm doubles when $m$ is big.

Next we recall a result  that proves that the $L^\infty$ norm falls quickly if the $L^1$ norm is bounded.
\begin{lemma}  \label{lem:Linfty-falls}
	There exists $C>0$ such that if $v \in L^1(D))$ then for any $t\in [0,1]$,\\
	\begin{equation} \label{eq:Linfty-falls}
		\int_{D}G(t,x-y)v(y)dy \leq C t^{-\frac{1}{2}} |v|_{L^1}.
	\end{equation}
\end{lemma}
\begin{proof}
	We proved in Lemma \ref{lem:kernel-estimates} that $|G(t,\cdot)|_{L^\infty} \leq C t^{-\frac{1}{2}}$. Therefore, for any $v \in L^1$, \eqref{eq:Linfty-falls} holds.
\end{proof}

\begin{lemma} \label{lem:doubling-probability}
	For any $p>6$ there exists a nonrandom constant $C_p>0$ and for any  $\e>0$ and $M>0$ there exists a nonrandom constant $m_0=m_0(\e,M)>0$ such that for any $n \in \mathbb{N}$,  and $m>m_0$
	\begin{align}
		&\Pro \left( |v(\rho_{n+1})|_{L^\infty} = 2^{m+1}  \Big| |v(\rho_n)|_{L^\infty} = 2^m  \right) \nonumber\\
		&\leq
		C_p M^{\frac{p}{2} - 3} \E\left( \int_{\rho_{n} }^{\rho_{n+1}} \int_D(\sigma(v(s,y)))^2 dyds \Bigg| |v(\rho_n)|_{L^\infty} = 2^m  \right)
	\end{align}
	Importantly,the constant $C_p$ is independent of $m > m_0$.
\end{lemma}
\begin{proof}
	Let $M>0$ and assume that $2^m = |v(\rho_{n} )|_{L^\infty}$. By the semigroup property of the heat semigroup, the mild solution for $t \in [0, \rho_{n+1}- \rho_{n} ]$, satisfies
	\begin{align}
		&v((t + \rho_{n}) , x) \nonumber\\
		&= \int_D G(t,x-y) v(\rho_{n} ,y)dy \nonumber\\
		&\qquad + \int_0^t \int_D G(t-s,x-y) f(v(s + \rho_n,y))dyds \nonumber\\
		&\quad+ \int_0^t \int_D G(t-s,x-y) \sigma(v(s + \rho_n,y))\mathbbm{1}_{\{s\leq \rho_{n+1}-\rho_n\}}  W(dy(ds + \rho_{n})) \nonumber\\
		&=: S_n(t,x) + K_n(t,x) + Z_n(t,x)	
	\end{align}

	By Lemma \ref{lem:Linfty-falls} and the fact that $|u(\rho_{n})|_{L^1}\leq M$ (remember that by definition $\rho_n \leq \tau^{\inf}_\e \wedge \tau^1_M$), it follows that for $t\in (0,1)$,
	\begin{equation} \label{eq:S-bound}
		|S_n(t)|_{L^\infty} \leq CM t^{-\frac{1}{2}}.
	\end{equation}
	Let $T_m = \frac{C^2M^2}{2^{2m-6}}$ so that $S_n(T_m) \leq 2^{m-3}$.
	We can bound
	\begin{equation}
		\sup_{t \leq T_m} \sup_{x \in D} |K_n(t,x)| \leq 2\pi T_m \e^{-\alpha} \leq \frac{C^2 M^2}{\e^\alpha 2^{2m-6}} ,
	\end{equation}
    because $f(v(s,y)) \leq \e^{-\alpha}$ for all $s \leq \rho_{n+1} \leq \tau^{\inf}_\e$.
	Choose $m_0 = m_0(\e,M)$ large enough so that for $m > m_0$,
	\begin{equation} \label{eq:K-bound}
			\sup_{t \leq T_m} \sup_{x \in D} |K_n(t,x)| \leq \frac{C^2 M^2}{\e^\alpha 2^{2m-2}} < 2^{m-3}.
	\end{equation}
	
	Theorem \ref{thm:stoch-conv} with
	\[\varphi(t,x) := \sigma( v(\rho_{n} + t,x) )\mathbbm{1}_{\{t\leq \rho_{n+1}-\rho_n\}}  \]
	 and the Chebyshev inequality guarantee that
	\begin{align}
		&\Pro \left(\sup_{t \leq T_m} \sup_{x \in D} |Z_n((t + \rho_n) ,x)| > 2^{m-2}  \Big| |v(\rho_n)|_{L^\infty} = 2^m\right) \nonumber\\
		&\leq 2^{-p(m-2)}\E\left(\sup_{t \leq T_m}  \sup_{x \in D} |Z_n((t + \rho_n) ,x)|^p  \Big| |v(\rho_n)|_{L^\infty}=2^m\right)\nonumber\\
		&\leq C 2^{-p(m-2)} T_m^{\left(\frac{p}{4} - \frac{3}{2}\right)} \nonumber\\
		&\qquad\times\E \left( \int_{\rho_{n}}^{(\rho_{n} + T_m) \wedge \rho_{n+1}} \int_D (\sigma(v(s,y)))^p dyds \Big| |v(\rho_n)|_{L^\infty} = 2^m\right) .
	\end{align}
	Because $|v(s,y)| \leq 2^{m+1}$ for $t \leq \rho_{n+1}$, our $\sigma$ growth restriction \ref{eq:sigma-growth} guarantees that $|\sigma(v(s,y))| \leq C(1 + 2^{\frac{3(m+1)}{2}}) \leq C 2^{\frac{3m}{2}}$. We bound 
	\begin{equation}
		|\sigma(v(s,y))|^p \leq|\sigma(v(s,y))|^{p-2}|\sigma(v(s,y))|^2 \leq C2^{\left(\frac{3pm}{2} - 3m\right)}|\sigma(v(s,y))|^2
	\end{equation} 
	and therefore
	\begin{align}
		&\Pro \left(\sup_{t \leq T_m} \sup_{x \in D} |Z_n((t + \rho_n) ,x)| > 2^{m-2}  \Big| |v(\rho_n)|_{L^\infty} = 2^m\right) \nonumber\\
		&\leq C 2^{-p(m-2)} 2^{\left(\frac{3pm}{2} - 3m\right)}T_m^{\left(\frac{p}{4} - \frac{3}{2}\right)} \nonumber\\
		&\qquad\times\E \left( \int_{\rho_{n}}^{(\rho_{n} + T_m) \wedge \rho_{n+1}} \int_D (\sigma(v(s,y)))^2 dyds \Big| |v(\rho_n)|_{L^\infty} = 2^m\right)
		\nonumber\\
		&\leq C  2^{\left(\frac{pm}{2} - 3m\right)}T_m^{\left(\frac{p}{4} - \frac{3}{2}\right)} \nonumber\\
		&\qquad\times\E \left( \int_{\rho_{n}}^{(\rho_{n} + T_m) \wedge \rho_{n+1}} \int_D (\sigma(v(s,y)))^2 dyds \Big| |v(\rho_n)|_{L^\infty} = 2^m\right)
	\end{align}
	$T_m$ is defined in such a way that $T_m^{\frac{1}{2}}2^m \leq C M$. This means that $$2^{\left(\frac{pm}{2} - 3m\right)}T_m^{\left(\frac{p}{4} - \frac{3}{2}\right)} \leq C M^{\frac{p}{2}-3} .$$ We also can bound $(\rho_{n} + T_m) \wedge \rho_{n+1} \leq \rho_{n+1}$. Therefore,
	\begin{align}
		&\Pro \left(\sup_{t \leq T_m} \sup_{x \in D} |Z_n((t + \rho_n) ,x)| > 2^{m-2}\Big| |v(\rho_n)|_{L^\infty} = 2^m \right) \nonumber\\
		&\leq C  M^{ \frac{p}{2} - 3} \E \left(\int_{\rho_{n}}^{\rho_{n+1}} \int_D (\sigma(v(s,y)))^2 dyds \Big| |v(\rho_n)|_{L^\infty} = 2^m\right) .
	\end{align}
	Finally we prove that if the event
	\begin{equation} \label{eq:event}
		\left\{\sup_{t \leq T_m} \sup_{x \in D} |Z_n((t + \rho_n) ,x)| \leq 2^{m-2} \right\}
	\end{equation} 
	occurs, then $|v(\rho_{n+1})|_{L^\infty}$ falls to $2^{m-1}$ before it can reach $2^{m+1}$.
	Specifically, because \eqref{eq:S-bound}--\eqref{eq:K-bound} prove that $|S_n(T_m)|_{L^\infty} + |K(T_m)|_{L^\infty} \leq 2^{m-2}$ it follows that $|u(\rho_n + T_m) |_{L^\infty} \leq 2^{m-1}$ on this event \eqref{eq:event}. On the other hand, $|S_n(t)|_{L^\infty} \leq 2^m$ for all $t \in [0,T_m]$ and it  follows that on this event \eqref{eq:event}, $\sup_{t \leq T_m} |u(\rho_n + t)| \leq 2^m + 2^{m-3} + 2^{m-2} < 2^{m+1}$. This implies that if the event \eqref{eq:event} occurs that $|u(\rho_n + t)|_{L^\infty}$ falls to the level  $2^{m-1}$ before it can rise to the level  $2^{m+1}$.  Therefore, for $m>m_0$
	\begin{align}
		&\Pro \left( |u(\rho_{n+1}  )|_{L^\infty} = 2^{m+1} \Big|  |v(\rho_{n}  )|_{L^\infty} = 2^m \right) \nonumber\\
		&\leq  \Pro \left(\sup_{t \leq T_m} \sup_{x \in D} |Z_n((t + \rho_n) \wedge \tau^1_M,x)| > 2^{m-2} \Big| |v(\rho_{n}  )|_{L^\infty} = 2^m \right)\nonumber\\
		&\leq
		C_p M^{\frac{p}{2} - 3} \E \left(\int_{\rho_{n} }^{\rho_{n+1}} \int_D (\sigma(v(s,y)))^2 dyds \Big| |v(\rho_n)|_{L^\infty} = 2^m\right).
	\end{align}
\end{proof}

\begin{proof}[Proof of Theorem \ref{thm:v-no-explode}]
	Fix  $M>0, \e>0$ and let $\rho_{n}$ be defined from \eqref{eq:rho-def}. Let $m_0$ be from \eqref{eq:K-bound}. We add up the conditional probabilities from Lemma \ref{lem:doubling-probability} to see that for any $n \in \mathbb{N}$,
	\begin{align}
			&\Pro \left(|v(\rho_{n + 1} ) |_{L^\infty}= 2 |v(\rho_{n})|_{L^\infty} \text{ and } |v(\rho_{n})|_{L^\infty} > 2^{m_0} \right)\nonumber\\
			&= \sum_{m=m_0}^\infty  C M^{\frac{3p}{2}-3} \E \left(\int_{\rho_n}^{\rho_{n+1}} \int_D (\sigma(v(s,y)))^2dyds \Big| |v(\rho_n)|_{L^\infty} = 2^m\right) \Pro \left(|v(\rho_n)|_{L^\infty} = 2^m \right)\nonumber\\
			&\leq C M^{\frac{3p}{2}-3} \E \left(\int_{\rho_n}^{\rho_{n+1}} \int_D (\sigma(v(s,y)))^2dyds \right).
	\end{align}
	Now add these probabilities with respect to $n$
	\begin{align} \label{eq:BC}
		&\sum_{n=1}^\infty \Pro \left(|v(\rho_{n + 1} ) |_{L^\infty}= 2 |v(\rho_{n})|_{L^\infty} \text{ and } |v(\rho_{n})|_{L^\infty} > 2^{m_0} \right)\nonumber\\
		&\leq C M^{\frac{3p}{2} - 3} \sum_{n=1}^\infty  \E \int_{\rho_{n} }^{\rho_{n+1}} (\sigma(v(s,y)))^2 dyds \nonumber \\
		&\leq C M^{\frac{3p}{2} - 3} \E \int_{0}^{\tau^{\inf}_\e \wedge \tau^1_M} (\sigma(v(s,y)))^2 dyds.
	\end{align}
	The last line is a consequence of the fact that all of the $\rho_n$ are defined to be smaller than $ \tau^{\inf}_\e \wedge \tau^1_M$. The right-hand side of \eqref{eq:BC} is proportional to the expectation of the quadratic variation from \eqref{eq:L3-bound}, which is finite. Therefore,
	
	\begin{equation}
		\sum_{n=1}^\infty \Pro \left(|v(\rho_{n + 1} ) |_{L^\infty} = 2|v(\rho_{n } ) |_{L^\infty} \text{ and }  |v(\rho_{n } ) |_{L^\infty} > 2^{m_0}\right) < +\infty.
	\end{equation}
	The Borel-Cantelli Lemma guarantees that with probability one, the events $\{|v(\rho_{n + 1} ) |_{L^\infty} = 2|v(\rho_{n }) |_{L^\infty} \text{ and } |v(\rho_{n } ) |_{L^\infty}> 2^{m_0}\}$
	only happen a finite number of times. This means that the $L^\infty$ norm cannot possibly explode before time $ \tau^{\inf}_\e \wedge \tau^1_M$ because the $L^\infty$ norm stops doubling when $m$ gets big.
%	\begin{equation}
%		\Pro\left(\sup_{t \leq T \wedge \tau^{\inf}_{\e} \wedge \tau^1_M } \sup_{x \in D}v(t,x) <+\infty\right)=1.
%	\end{equation}
	This proves that for any $\e>0$ and $M<\infty$,
	\begin{equation} \label{eq:prob-eps-M}
		\Pro \left( (\tau^{\inf}_\e \wedge \tau^1_M) < \tau^\infty_\infty\right) =1.
	\end{equation}
	Next, we argue via Proposition \ref{prop:v-positive} and Lemma \ref{lem:L1} that for arbitrary $T>0$ and small enough $\e>0$ and large enough $M>0$, the stopping times $\tau^{\inf}_\e$ and $\tau^1_M$ are both larger than $T$ with high probability.
	Let $\eta\in (0,1)$ and $T>0$ be arbitrary. Proposition \ref{prop:v-positive} implies that there exists $\e>0$ small enough so that
	\begin{equation} \label{eq:eps-small-rare}
		\Pro\left( \tau^{\inf}_\e < (T \wedge \tau^\infty_\infty) \right)=\Pro\left(\inf_{t \in [0,T\wedge \tau^\infty_\infty)} \inf_{x \in D} v(t,x) \leq \e \right) < \frac{\eta}{2}.
	\end{equation}
	With this choice of $\e>0$, we estimate the probability that the $L^1$ norm of $v(t,x)$ is large. For $M>0$ to be chosen later,
	\begin{align} 
		& \Pro\left( \tau^1_M < (T \wedge \tau^\infty_\infty) \text{ or } \tau^{\inf}_\e < (T \wedge \tau^\infty_\infty)\right) \nonumber\\
		&=\Pro\left( \sup_{t \in [0,(T \wedge \tau^\infty_\infty)]} \int_Dv(t,x)dx > M \text{ or } \tau^{\inf}_\e < (T \wedge \tau^\infty_\infty)\right)\nonumber\\
		&\leq \Pro\left( \tau^{\inf}_\e <(T \wedge \tau^\infty_\infty)\right) \nonumber\\
		&\qquad
		+\Pro\left(\sup_{t \in [0,T \wedge \tau^\infty_\infty]} \int_Dv(t,x)dx > M\text{ and } \tau^{\inf}_\e \geq (T\wedge \tau^\infty_\infty)\right)\nonumber\\
		&\leq \frac{\eta}{2}
		+ \Pro\left(\sup_{t \in [0,T \wedge \tau^\infty_\infty \wedge \tau^{\inf}_\e]} \int_Dv(t,x)dx > M\right)
	\end{align}
	The last line follows from \eqref{eq:eps-small-rare} and the fact that on the event $\{\tau^{\inf}_\e \geq T \wedge \tau^\infty_\infty\}$,  $T \wedge \tau^\infty_\infty = T \wedge \tau^\infty_\infty \wedge \tau^{\inf}_\e$.
	By Lemma \ref{lem:L1}, we can choose $M>0$ large enough so that
	\begin{equation} \label{eq:big-L1-rare}
		\Pro\left( \tau^1_M < (T \wedge \tau^\infty_\infty) \text{ or } \tau^{\inf}_\e < (T \wedge \tau^\infty_\infty)\right) < \eta.
	\end{equation}
	Therefore, with these choices of $\e>0$ and $M>0$,  \eqref{eq:eps-small-rare} and \eqref{eq:big-L1-rare} imply
	\begin{equation}
		\Pro \left( \tau^{\inf}_\e \geq T \wedge \tau^\infty_\infty \text{ and } \tau^1_M \geq T \wedge \tau^\infty_\infty \right) \geq 1 - \eta.
	\end{equation}
	This combined with \eqref{eq:prob-eps-M} implies that
	\begin{equation}
		\Pro(T < \tau^\infty_\infty) > 1 - \eta.
	\end{equation}
	The choice of $\eta>0$ was arbitrary. Therefore,
	\begin{equation}
		\Pro(T < \tau^\infty_\infty) = 1.
	\end{equation}
	This is true for arbitrary $T>0$ and therefore,
	\begin{equation}
		\Pro\left(\tau^\infty_\infty=\infty \right)=1
	\end{equation}
	and $v(t,x)$ cannot explode in finite time.
	
\end{proof}

\bibliography{superlinear}

\end{document}